\documentclass[11pt]{article}

\usepackage[T2A]{fontenc}
\usepackage[cp1251]{inputenc}
\usepackage{amsthm}
\usepackage{amsfonts}
\usepackage{eufrak}
\usepackage{amssymb}
\usepackage{amsmath}
\usepackage{cite}

\makeatletter
\renewcommand{\@seccntformat}[1]{\csname the#1\endcsname.}
\makeatother
\usepackage{amsthm}
\usepackage{bm}
\begin{document}
\newtheoremstyle{mytheorem}
  {\topsep}   
  {\topsep}   
  {\itshape}  
  {}       
  {\bfseries} 
  {. }         
  {5pt plus 1pt minus 1pt} 
  { }          
\newtheoremstyle{myremark}
  {\topsep}   
  {\topsep}   
  {\upshape}  
  {}       
  {\bfseries} 
  {. }         
  {5pt plus 0pt minus 1pt} 
  {}          
\theoremstyle{mytheorem}
\newtheorem{theorem}{Theorem}[section]
 \newtheorem{theorema}{Theorem}
 \renewcommand{\thetheorema}{\Alph{theorema}}
\newtheorem{proposition}[theorem]{Proposition} 
 \newtheorem{lemma}[theorem]{Lemma}
\newtheorem{corollary}[theorem]{Corollary}
\newtheorem{definition}[theorem]{Definition}
\theoremstyle{myremark}
\newtheorem{remark}[theorem]{Remark}
\noindent This article is accepted for publishing in

\noindent Journal of Mathematical Analysis and Applications

\vskip 1 cm

\noindent{\bf Characterization of probability distributions on some  
locally \\ compact Abelian groups containing an element of order $2$}

\bigskip

\noindent{\bf Gennadiy Feldman} 

\bigskip

\noindent {\bf  Abstract.} {The well-known Heyde  theorem    
characterizes the Gaussian distributions on the real line by the symmetry of the conditional distribution of one linear form of independent random variables 
given another. We generalize 
this theorem to groups of the form $\mathbb{R}\times F$,
where $F$ is a finite Abelian group such that its 2-component is isomorphic to the  additive group of the integers 
modulo $2$.
In so doing, coefficients of the linear forms  are arbitrary topological 
automorphisms of the group. Previously, a similar result was proved in 
the case when the group $F$ contains no elements of order 2. 
The presence of an element of order 2 in $F$
 leads to the fact that a new class 
of probability distributions is characterized}.

\bigskip


\bigskip


\section { Introduction}

By the well-known Skitovich--Darmois theorem the Gaussian distribution  
on the real line is characterized by the independence of two linear forms
of independent random variables. A similar theorem was proved by C.~C.~Heyde,
where instead of the independence the symmetry of the conditional distribution 
of one linear form  given another was considered 
(\!\!\cite{He}, see also \cite[Theorem 13.4.1]{KaLiRa}). 
 A number of works  
were devoted to generalizing of   
Heyde's theorem to various locally compact Abelian 
groups (see e.g. \cite{F2023, JFAA2021, Rima, Fe4, Fe20bb, My2, Fe2015a, Fe6, M2023, POTA, FeTVP1, {F_solenoid}}  and also   
  \cite[Chapter IV]{Febooknew}, where one can find additional references).  
  In so doing, coefficients of the linear forms  are   topological 
  automorphisms of a group, as a rule satisfying  certain conditions.
 
 In article \cite{Rima} Heyde's theorem for two independent random variables 
was generalized to groups of the form $\mathbb{R}\times G$, where $G$ is a finite Abelian group containing no elements of order 2. The presence of  elements 
of order 2 in a group plays an exceptional role in generalizing Heyde's theorem.

Let $X$ be a finite Abelian group. For each prime number $p$ denote  by $X_p$ 
the subgroup of $X$ consisting of all elements of $X$ whose order is a power of $p$.
The subgroups  $X_p$ are called $p$-components of $X$.
The aim of the article is to generalize Heyde's theorem  
to groups of the form $\mathbb{R}\times F$, where $F$ is a finite Abelian group 
 such that its 2-component is isomorphic to the additive group of 
 the integers modulo 2.
It turns out that in this case a new class 
of probability distributions is characterized. It is very likely that this class of distributions arise when we prove other characterization theorems on groups with similar structures and also in studying of the arithmetic of some semigroups of probability distributions on groups. Let us note one more circumstance. Usually, the proof of a characterization theorem on a locally compact Abelian group reduces to solving some functional equation for characteristic functions of the random variables on the  character group of the original group. In other words, we   deal  only with the characteristic functions, and not with the distributions. In our case, a new method is used. It consists of simultaneous consideration of both the characteristic functions and the distributions themselves. This method may also have applications in the study of other characterization problems.

It should be emphasized that despite the probabilistic formulation of the problem, 
the proof is purely analytical. Mainly methods of abstract 
harmonic analysis are used, as well as some facts of 
complex analysis.

Consider a locally compact Abelian  (LCA) group $X$ and denote by $Y$ its character
group. Take  $x \in X$ and denote by  $(x,y)$ the value of a 
character $y \in Y$ at the element  $x$.   Let $H$ be a closed subgroup of 
the group $Y$.   Denote by
 $A(X, H) = \{x \in X: (x, y) =1$ \mbox{for all } $y\in
H\}$
the annihilator of $H$. 
The group
of all topological automorphisms of the group $X$ we denote by ${\rm Aut}(X)$.
The identity 
automorphism of a group is denoted by $I$.   Let $\alpha\in{\rm Aut}(X)$ and  
$K$ be a closed subgroup of   $X$. Suppose  $\alpha(K)=K$, i.e., the restriction  
of  
$\alpha$ to $K$ is a topological automorphism of the group $K$. 
This restriction is denoted by $\alpha_{K}$. A closed subgroup $K$
of the group $X$ is called characteristic if $\alpha(K)=K$ for any 
$\alpha\in{\rm Aut}(X)$.
Let
 $\alpha\in{\rm Aut}(X)$. The adjoint automorphism
$\widetilde\alpha\in{\rm Aut}(Y)$
is defined by the formula $(x, \widetilde\alpha y)=(\alpha x, y)$ for all $x\in X$, 
$y\in Y$. It should be noted that $\alpha\in{\rm Aut}(X)$ if and only if 
$\widetilde\alpha\in{\rm Aut}(Y)$.  Set $X^{(2)}=\{2x: x\in X\}$. 
 The additive groups of real numbers and  
 the integers 
modulo $2$ are denoted by $\mathbb{R}$ and 
$\mathbb{Z}(2)=\{0, 1\}$ respectively. Denote by $\mathbb{C}$ the 
complex plane.

Denote by $\mathrm{M}^1(X)$ the set of all probability
distributions on the group $X$. Let $\mu, \nu\in \mathrm{M}^1(X)$. The convolution
$\mu*\nu\in \mathrm{M}^1(X)$ is defined by the formula
$$
\mu*\nu(E)=\int\limits_{X}\mu(E-x)d \nu(x)
$$
for any Borel subset $E$ of $X$.
Then the set $\mathrm{M}^1(X)$ is the semigroup with respect to the convolution. 
Take $\mu\in\mathrm{M}^1(X)$. 
The characteristic function (Fourier transform) 
$\hat\mu(y)$ of the distribution  $\mu$
is defined by the formula
$$
\hat\mu(y) =
\int\limits_{X}(x, y)d \mu(x), \quad y\in Y.$$
For   signed measures, the convolution and the   
characteristic function are defined similarly. We recall that 
$\widehat{\mu*\nu}(y)=\hat\mu(y)\hat\nu(y)$.
     The support of $\mu$ is denoted by $\sigma(\mu)$.   
For $x\in X$ denote by $E_x$  
   the degenerate distribution
 with the support at the element $x$. Denote by $\Gamma(\mathbb{R})$ the 
 set of Gaussian distributions on the group $\mathbb{R}$ 
 (by Gaussian distributions on the real line we mean both 
 ordinary normal distributions and degenerate distributions).
 
\section { Lemmas}

To prove the main theorem we need some lemmas.

\begin{lemma}[\kern-0.55ex{\protect \cite[Lemma 16.1]{Fe5}}]\label{le1}
Assume that $X$ is a second countable LCA group,   $Y$ is its character group,   
and   $\alpha\in{\rm Aut}(X)$.
Let
$\xi_j$, $j=1, 2$,  be independent random variables with values in the group 
$X$  and distributions $\mu_j$.  Then the following statements are equivalent:
\renewcommand{\labelenumi}{\rm(\roman{enumi})}
\begin{enumerate}
  
\item	

the conditional distribution of the 
linear form $L_2 = \xi_1 + \alpha\xi_2$ given $L_1 = \xi_1 + \xi_2$ is 
symmetric; 

\item

the characteristic functions
 $\hat\mu_j(y)$ satisfy the equation
\begin{equation}\label{03.01.1}
\hat\mu_1(u+v)\hat\mu_2(u+\widetilde\alpha v)=
\hat\mu_1(u-v)\hat\mu_2(u-\widetilde\alpha v), \quad u, v \in Y.
\end{equation}
\end{enumerate}
\end{lemma}
Thanks to Lemma \ref{le1}, the description of distributions of independent random variables with values in the group $X$ which are characterized by
the symmetry of the conditional distribution of one linear form 
of these random variables given another is 
reduced to the description of solutions of functional equation (\ref{03.01.1}) in the class of characteristic functions on the group $Y$.

The following lemma is an analogue of Heyde's theorem for finite Abelian groups containing no elements of order  $2$.
 
\begin{lemma}[\kern-0.55ex{\protect \cite[\S 2]{Rima}, see also \cite[Theorem 9.11]{Febooknew}}]\label{le2}  Assume that $G$ is a finite Abelian group, $G$ contains no elements of order  $2$,  and   $\alpha$ is an automorphism of    
$G$. Put $K={\rm Ker}(I+\alpha)$. Let $\xi_j$, $j=1, 2$, be
independent random variables with values in the group  $G$ and distributions
$\mu_j$ with nonvanishing characteristic functions. Suppose that the conditional  distribution of the linear form $L_2 = \xi_1 + \alpha\xi_2$ given $L_1 = \xi_1 + \xi_2$ is symmetric.
Then $\mu_j=\omega*E_{g_j}$, where $\omega$ is a distribution  supported in $K$, 
$g_j\in G$. Moreover,  if  $\zeta_j$ are independent identically distributed random variables with values in the group  $G$ and distribution
    $\omega$, then the conditional distribution of the linear form
$M_2 = \zeta_1 +\alpha\zeta_2$ given $M_1 = \zeta_1 + \zeta_2$  is symmetric.
\end{lemma}
The following  assertion is well known. We formulate it in the form of a lemma. 
\begin{lemma}\label{le6}  Assume that $X$ is a LCA group, $Y$ is its character group,  $H$ is a closed subgroup of  $Y$, and  $\mu$ is a 
distribution on $X$.
 If $\hat\mu(y)=1$ for all $y\in H$, then $\sigma(\mu)\subset A(X, H)$.    
\end{lemma}

For what follows, we will need the following notation. Elements of the group   
$\mathbb{R}\times \mathbb{Z}(2)$ are  denoted by 
$(t,   m)$, where $t\in \mathbb{R}$,   $m\in \mathbb{Z}(2)$.
Elements of the character group   of the group $\mathbb{R}\times \mathbb{Z}(2)$
which is topologically isomorphic 
 to $\mathbb{R}\times \mathbb{Z}(2)$  
are  denoted by $(s, n)$, where $s\in \mathbb{R}$,   $n\in \mathbb{Z}(2)$. 
 
\begin{lemma}[\kern-0.55ex{\protect \cite[Lemma 4.1]{F_solenoid}, see also \cite[Lemma 11.1]{Febooknew}}]
 \label{le19.01.1}  
     Assume that a function $\phi(s, n)$ on  
the character 
group of the group $\mathbb{R}\times \mathbb{Z}(2)$ is represented in the form
\begin{equation}\label{21.01.1}  
\phi(s, n) = \begin{cases}\exp\{-\sigma s^2+i{\mathfrak{m}} s\}, &\text{\ if\ }\ 
\ s\in \mathbb{R}, 
\ n=0,\\
\varkappa\exp\{-\sigma' s^2+i\mathfrak{m}'s\}, &\text{\ if\ }\ \ s\in \mathbb{R}, 
\   n=1,
\\
\end{cases}
\end{equation} 
where $\sigma\ge 0$,  $\sigma'\ge 0$ and  
$\mathfrak{m}, \mathfrak{m}', \varkappa\in \mathbb{R}$. 
Then there is a signed measure 
$\mu$ on the 
group $\mathbb{R}\times \mathbb{Z}(2)$ such that
$\hat\mu(s, n)=\phi(s, l)$, $s\in \mathbb{R}$,   $n\in \mathbb{Z}(2)$.
Moreover, the following statements are equivalent:
\renewcommand{\labelenumi}{\rm(\roman{enumi})}
\begin{enumerate}
\item	
$\mu$ is a distribution; 
\item
either 
\begin{equation}\label{08.01.8}
0<\sigma'<\sigma, \quad 
0<|\varkappa|\le\sqrt\frac{\sigma'}{\sigma}\exp\left\{-\frac{(\mathfrak{m}-\mathfrak{m}')^2}
{4(\sigma-\sigma')}\right\} 
\end{equation}
or
\begin{equation}\label{08.01.9}
\sigma=\sigma', \quad\mathfrak{m}=\mathfrak{m}', \quad  |\varkappa|\le 1.
\end{equation}
\end{enumerate}
If $(\ref{08.01.9})$ is satisfied, then $\mu\in \Gamma(\mathbb{R})*\textup{M}^1(\mathbb{Z}(2))$. 
\end{lemma}
\begin{definition}[\kern-0.55ex{\protect \cite[Definition 2.1]{POTA}}]\label{de1} Let 
$\mu\in \textup{M}^1(\mathbb{R}\times \mathbb{Z}(2))$. We say that    
$\mu\in\Theta$  if $\hat\mu(s, n)=\phi(s, n)$, $s\in \mathbb{R}$,   
$n\in \mathbb{Z}(2)$, where  the function   $\phi(s, n)$ 
is represented in the form 
$(\ref{21.01.1})$
  and either $(\ref{08.01.8})$ or 
$(\ref{08.01.9})$ are satisfied.
\end{definition}

Let $\alpha\in {\rm Aut}(\mathbb{R}\times \mathbb{Z}(2))$. 
It is obvious that
 $\alpha$ is of the form
 $\alpha(t, m)=(a t, m)$, $t\in \mathbb{R}$,   $m\in \mathbb{Z}(2)$,  
 where $a\in \mathbb{R}$, $a\ne 0$.
  Observe also that 
$\widetilde \alpha$ is of the form 
$\widetilde \alpha(s, n)=(a s, n)$, $s\in \mathbb{R}$,   $n\in \mathbb{Z}(2)$.
The following lemma is an analogue of Heyde's theorem for the 
group $\mathbb{R}\times \mathbb{Z}(2)$.

\begin{lemma}[\kern-0.55ex{\protect \cite[Theorem 2.1]{POTA}, 
see also \cite[Theorem 11.6]{Febooknew}}]\label{le3}   Consider a topological automorphism $\alpha$  of the group 
$\mathbb{R}\times \mathbb{Z}(2)$ of the form 
$\alpha(t, m)=(at, m)$, $t\in \mathbb{R}$, $m\in\mathbb{Z}(2)$,  where 
 $a\ne -1$.
Let $\xi_j$, $j=1, 2$, be
independent random variables with values in    
$\mathbb{R}\times \mathbb{Z}(2)$ and distributions $\mu_j$ with 
nonvanishing characteristic functions.
If the conditional distribution of the linear form
$L_2 = \xi_1 + \alpha\xi_2$ given 
$L_1 = \xi_1 + \xi_2$ is symmetric, then 
$\mu_{j}\in\Theta$, $j=1, 2$.
\end{lemma}

 Denote by 
$(t,   m, g)$, where $t\in \mathbb{R}$,   $m\in \mathbb{Z}(2)$, $g\in G$,
elements of the group $\mathbb{R}\times \mathbb{Z}(2)\times G$, where $G$ is a finite Abelian group. 
Take $\sigma> 0$ and $\mathfrak{m}\in\mathbb{R}$.  Denote by $\gamma_{\sigma, \mathfrak{m}}$ the Gaussian distribution on the group
$\mathbb{R}$ with the density 
\begin{equation}\label{y5}
\rho_{\sigma, \mathfrak{m}}(t)=\frac{1}{2\sqrt{\pi \sigma}}
\exp\left\{-\frac{(t-\mathfrak{m})^2}{4\sigma}\right\},
\quad t\in \mathbb{R},
\end{equation}
with respect to the normalized Lebesgue measure.
Then the characteristic function of $\gamma_{\sigma, \mathfrak{m}}$ is of the form
\begin{equation}\label{26.01.1}
\hat\gamma_{\sigma, \mathfrak{m}}(s)=\exp\{-\sigma s^2+i\mathfrak{m} s\}, 
\quad s\in \mathbb{R}.
\end{equation}
\begin{lemma}\label{le17.2} Consider the group $\mathbb{R}\times \mathbb{Z}(2)\times G$, where $G$ is a finite Abelian group. Let $0<\sigma'<\sigma$   and   
$p=(0, 1, 0)$ be the element of order $2$ of the  subgroup
 $\mathbb{Z}(2)$. Consider the signed measure
$$
\lambda=\frac{1}{2}\left(\gamma_{\sigma, \mathfrak{m}}+
\gamma_{\sigma', \mathfrak{m}'}\right)+\frac{1}{2}\left(\gamma_{\sigma, \mathfrak{m}}-
\gamma_{\sigma', \mathfrak{m}'}\right)*E_p 
$$
on the subgroup   
$\mathbb{R}\times \mathbb{Z}(2)$ and the distribution 
$$
\tau=\sum_{g_i\in G} a_iE_{g_i}+\sum_{g_i\in G} b_iE_{g_i+p},
$$
where $a_i\ge 0$,  $b_i\ge 0$, on the subgroup $\mathbb{Z}(2)\times G$. The convolution
$\lambda*\tau$ is a distribution if and only if the inequalities
$$
\left|\frac{a_i-b_i}{a_i+b_i}\right|\le\sqrt\frac{\sigma'}{\sigma}\exp\left\{-\frac{(\mathfrak{m}-\mathfrak{m}')^2}
{4(\sigma-\sigma')}\right\} 
$$
are true for all $a_i$, $b_i$ such that $a_i+b_i>0$.
\end{lemma}
\begin{proof}  We split the proof of the lemma into two steps. 

1. Take $\varkappa>0$ and verify that the signed measure
$$\gamma_{\sigma, \mathfrak{m}}-\varkappa\gamma_{\sigma', \mathfrak{m}'}
$$
is a measure if and only if the inequality
$$
\varkappa\le\sqrt\frac{\sigma'}{\sigma}
\exp\left\{-\frac{(\mathfrak{m}-\mathfrak{m}')^2}
{4(\sigma-\sigma')}\right\}  
$$
is fulfilled. For the proof, we follow Lemma 4.1 in 
\cite{F_solenoid} (see also \cite[Lemma 11.1]{Febooknew}).

In view of (\ref{y5}), the signed measure 
$\gamma_{\sigma, \mathfrak{m}}-\varkappa\gamma_{\sigma', \mathfrak{m}'}$ is 
a measure if and only if the inequality
$$
\frac{1}{2\sqrt{\pi \sigma}}
\exp\left\{-\frac{(t-\mathfrak{m})^2}{4\sigma}\right\}
-\frac{\varkappa}{2\sqrt{\pi \sigma'}}
\exp\left\{-\frac{(t-\mathfrak{m}')^2}{4\sigma'}\right\}\ge 0
$$
holds for all $t\in \mathbb{R}$.
This inequality is equivalent to the following
\begin{equation}\label{25.01.2}
\varkappa\le\sqrt\frac{\sigma'}{\sigma}
\exp\left\{-\frac{(t-\mathfrak{m})^2}{4\sigma}
+\frac{(t-\mathfrak{m}')^2}{4\sigma'}\right\}, 
\quad t\in \mathbb{R}.
\end{equation}
It is easy to see that the minimum of the function in the right-hand 
 side of inequality (\ref{25.01.2})
 is equal to  
 $$
 \sqrt\frac{\sigma'}{\sigma}
 \exp\left\{-\frac{(\mathfrak{m}-\mathfrak{m}')^2}{4(\sigma-\sigma')}\right\}.
 $$  
The required assertion follows from this.

2. Taking into account that $E_p*E_p=E_{2p}=E_0$, we have
\begin{multline}\label{19.01.1}
\lambda*\tau=\left(\frac{1}{2}\left(\gamma_{\sigma, \mathfrak{m}}+
\gamma_{\sigma', \mathfrak{m}'}\right)+\frac{1}{2}\left(\gamma_{\sigma, \mathfrak{m}}-
\gamma_{\sigma', \mathfrak{m}'}\right)*E_p\right)*
\left(\sum_{g_i\in G} a_iE_{g_i}+\sum_{g_i\in G} b_iE_{g_i+p}\right)\\=\frac{1}{2}\sum_{g_i\in G}\left((a_i+b_i)
\gamma_{\sigma, \mathfrak{m}}-(b_i-a_i)\gamma_{\sigma', \mathfrak{m}'}\right)*E_{g_i}
\\+\frac{1}{2}\sum_{g_i\in G}\left((a_i+b_i)
\gamma_{\sigma, \mathfrak{m}}-(a_i-b_i)\gamma_{\sigma', \mathfrak{m}'}\right)*E_{g_i+p}. 
\end{multline}
Fix an element $g_i\in G$ and take a Borel subset $E$ of $\mathbb{R}$. 
Suppose $a_i+b_i>0$. It follows from (\ref{19.01.1}) that
$$
(\lambda*\tau)(E\times\{g_i\})=\frac{a_i+b_i}{2}\left(
\gamma_{\sigma, \mathfrak{m}}-\frac{b_i-a_i}{a_i+b_i}\gamma_{\sigma', \mathfrak{m}'}\right)(E)
$$
and
$$
(\lambda*\tau)(E\times\{g_i+p\})=\frac{a_i+b_i}{2}\left(
\gamma_{\sigma, \mathfrak{m}}-\frac{a_i-b_i}{a_i+b_i}\gamma_{\sigma', \mathfrak{m}'}\right)(E).
$$
The statement of the lemma follows from the assertion proved in item 1.
\end{proof}

Let $K$ be a LCA group. 
Consider the group   $X=\mathbb{R}\times K$. 
 The character group $Y$ of the group $X$ is topologically isomorphic 
 of the group $\mathbb{R}\times L$, where $L$ is the
character group of the group $K$. We denote by $y=(s,   l)$, where $s\in \mathbb{R}$,   $l\in L$, elements of the group $Y$.
\begin{lemma}[\kern-0.55ex{\protect \cite[Lemma 6.9]{Fe9}}]\label{le4}   
Consider a group $\mathbb{R}\times K$, where $K$ is a LCA group. 
Take $\mu\in \textup{M}^1(\mathbb{R}\times K)$  
	and suppose that  $\hat\mu(s, 0)$ is an entire
	function in $s$. This implies that $\hat\mu(s, l)$ is an entire function in
	$s$ for each fixed  $l\in L$ and the inequality
\begin{equation}\label{14.03.1}
\max_{s\in \mathbb{C}, \ |s|\le r}|\hat\mu(s, l)|\le 
	\max_{s\in \mathbb{C}, \ |s|\le r}|\hat\mu(s, 0)|, 
	\quad l\in L,
\end{equation}	
holds.
\end{lemma}
\begin{lemma}[\kern-0.55ex{\protect \cite[Remark 4.1]{F_solenoid}, 
see also \cite[Proposition 11.9]{Febooknew}}]\label{le5}   
Assume that $X$ is a second countable LCA group containing exactly 
one element  
of order  $2$. Denote by $P$ the subgroup of $X$ generated by this element.  
Let $\xi_j$, $j=1, 2$, be independent random variables
with values in the group
       $X$  and distributions
  $\tau_j$ with nonvanishing characteristic functions.
Then the following statements are equivalent:
\renewcommand{\labelenumi}{\rm(\roman{enumi})}
\begin{enumerate}
  
\item	

the conditional distribution of the 
linear form $L_2=\xi_1-\xi_2$
 given $L_1 = \xi_1 + \xi_2$ is 
symmetric; 

\item

either $\tau_1=\tau_2*\delta_2$ or   $\tau_2=\tau_1*\delta_1$, 
 where  $\delta_j\in \textup{M}^1(P)$, $j=1, 2$.
\end{enumerate}  
\end{lemma}

\section { Main theorem}

Assume that $X$ is a second countable LCA group  and 
 $\alpha_j, \beta_j$, $j=1, 2$, are   topological automorphisms  of   $X$.
Let
$\xi_j$  be independent random variables with values in
the group $X$.  
Suppose that the conditional distribution of the linear form 
 $L_2 = \beta_1\xi_1 + \beta_2\xi_2$ given 
 $L_1 = \alpha_1\xi_1 + \alpha_2\xi_2$ is symmetric.
 It is easy to see that  studying the possible distributions of $\xi_j$, 
 we can assume that 
 $L_1 = \xi_1 + \xi_2$, $L_2 =  \xi_1 + \alpha\xi_2$, 
 where $\alpha\in{\rm Aut}(X)$.
 
 Consider a group $X$ of the form $X=\mathbb{R}\times F$,
where $F$ is a finite Abelian group such that its 2-component is isomorphic to 
$\mathbb{Z}(2)$.
Since a finite Abelian group is isomorphic to a direct product of
its $p$-components, the group $F$ is isomorphic to a group of the form $\mathbb{Z}(2)\times G$, 
where $G$ is a finite Abelian group containing no elements of order 2. So, we will 
 assume, without loss of generality, that $X=\mathbb{R}\times \mathbb{Z}(2)\times G$.
 Let $\alpha$ be a topological automorphism of 
the group $X$. Since $\mathbb{R}$, $\mathbb{Z}(2)$, and 
$G$ are characteristic subgroups of the group $X$,   $\alpha$ acts on the elements of the group $X$ as follows: $\alpha(t, m, g)=(at, m, \alpha_Gg)$, $t\in \mathbb{R}$, $m\in\mathbb{Z}(2)$, $g\in G$, where $a\in \mathbb{R}$, $a\ne 0$, and we will write $\alpha=(a, I, \alpha_G)$. The character group $Y$ of the group $X$ is topologically isomorphic of the group $\mathbb{R}\times \mathbb{Z}(2)\times H$, where $H$ is the
character group of the group $G$. We denote by $y=(s, n, h)$, where $s\in \mathbb{R}$, $n\in\mathbb{Z}(2)$, $h\in H$, elements of the group $Y$.
The value of a character $y=(s, n, h)$ at an element
$x=(t, m, g)$   is defined by the formula $$((t, m, g),(s, n, h))= 
e^{its}(-1)^{mn}(g,h).$$ 
 
The aim of the article is to prove the following group analogue of Heyde's 
theorem.
\begin{theorem}\label{th1} Assume that $X=\mathbb{R}\times \mathbb{Z}(2)\times G$, 
where $G$ is a finite Abelian group containing no elements of order  $2$.
Let  $\alpha=(a, I, \alpha_G)$ be a topological automorphism of 
the group $X$.  Set $K={\rm Ker}(I+\alpha_{G})$.
Let $\xi_j$, $j=1, 2$, be
independent random variables with values in   $X$ and distributions
$\mu_j$ with nonvanishing characteristic functions. Assume that the conditional  distribution of the linear form $L_2 = \xi_1 + \alpha\xi_2$ given $L_1 = \xi_1 + \xi_2$ is symmetric. Then the following statements are true.
\renewcommand{\labelenumi}{\rm{\Roman{enumi}.}}
\begin{enumerate}
  
\item	

If  $a \ne-1$, then $\mu_j=\gamma_j*\omega_j*E_{g_j}$, where $\gamma_j$ are distributions of the class $\Theta$ on the subgroup   
$\mathbb{R}\times \mathbb{Z}(2)$, $\omega_j$ are distributions on the subgroup    
$\mathbb{Z}(2)\times K$ and either 
$\omega_1=\omega_2*\vartheta_2$ or   $\omega_2=\omega_1*\vartheta_1$, 
 where  $\vartheta_j\in \mathrm{M}^1(\mathbb{Z}(2))$, $g_j\in G$. 

\item

If  $a=-1$, then 
$\mu_j=\omega_j*E_{x_j}$, where  $\omega_j$ are distributions on  
the subgroup $\mathbb{R}\times\mathbb{Z}(2)\times K$  and either 
$\omega_1=\omega_2*\vartheta_2$ or   $\omega_2=\omega_1*\vartheta_1$, 
 where  $\vartheta_j\in \mathrm{M}^1(\mathbb{Z}(2))$, $g_j\in G$, $j=1, 2$. 
\end{enumerate}
\end{theorem}
\begin{proof} We split the proof of the theorem into some steps.
 
 1.  By Lemma \ref{le1}, the characteristic functions $\hat\mu_j(y)$ satisfy equation (\ref{03.01.1}) which takes the form
\begin{multline}
\label{03.01.2}
\hat\mu_1(s_1+s_2, n_1+n_2, h_1+h_2)\hat\mu_2(s_1+a s_2, n_1+n_2, h_1+\widetilde\alpha_Gh_2) 
\\=
\hat\mu_1(s_1-s_2, n_1+n_2, h_1-h_2)\hat\mu_2(s_1-a s_2, n_1+n_2, h_1-\widetilde\alpha_Gh_2), \quad s_j\in \mathbb{R}, \  n_j\in\mathbb{Z}(2),  \  h_j\in H.
\end{multline}
Substitute $s_1=s_2=0$ and $n_1=n_2=0$ in equation (\ref{03.01.2}).
Taking into account Lemma  \ref{le1}  and reasoning as in the proof of Theorem 3.1 in \cite{Rima}, it follows from the obtained equation and  
Lemma  \ref{le2}  that there exist elements $g_1, g_2\in G$  such that 
if $\theta_j=\mu_j*E_{-g_j}$ and $\eta_j$ are independent random variables with values in the group $X$ and distributions  $\theta_j$, then the conditional distribution 
of the linear form
$N_2=\eta_1 + \alpha\eta_2$ given $N_1=\eta_1 + \eta_2$ is  symmetric. 
Moreover, $\sigma(\theta_j)\subset\mathbb{R}\times\mathbb{Z}(2)\times K$, $j=1, 2$. 
This means that independent random variables $\eta_j$ take values in $\mathbb{R}\times\mathbb{Z}(2)\times K$.

By the condition of the theorem,  $K={\rm Ker}(I+\alpha_{G})$. Hence
$\alpha_{G}g=-g$ for all $g\in K$, i.e., 
 the restriction of the automorphism   $\alpha_{G}$  to the 
 subgroup $K$ coincides with $-I$. From the above it follows that
the proof of the theorem reduces to the case when the topological automorphism   
$\alpha$   is of the form $\alpha=(a, I, -I)$. 

2. Assume $a\ne -1$. In view of $\alpha=(a, I, -I)$, we can rewrite equation (\ref{03.01.2}) 
in the form
\begin{multline}
\label{03.01.4}
\hat\mu_1(s_1+s_2, n_1+n_2, h_1+h_2)\hat\mu_2(s_1+a s_2, n_1+n_2, h_1-h_2) 
\\=
\hat\mu_1(s_1-s_2, n_1+n_2, h_1-h_2)\hat\mu_2(s_1-a s_2, n_1+n_2, h_1+h_2), 
\quad s_j\in \mathbb{R}, \  n_j\in\mathbb{Z}(2),  \  h_j\in H.
\end{multline}
Substituting $h_1=h_2=0$ in equation (\ref{03.01.4}), we obtain
\begin{multline}
\label{03.01.5}
\hat\mu_1(s_1+s_2, n_1+n_2, 0)\hat\mu_2(s_1+a s_2, n_1+n_2, 0) 
\\=
\hat\mu_1(s_1-s_2, n_1+n_2, 0)\hat\mu_2(s_1-a s_2, n_1+n_2, 0), 
\quad s_j\in \mathbb{R}, \  n_j\in\mathbb{Z}(2).
\end{multline}

Taking into account Lemma \ref{le1}, it follows from
   equation (\ref{03.01.5}) and  Lemma \ref{le3} that the characteristic functions
$\hat\mu_j(s, n, 0)$ are represented in the form
\begin{equation}\label{03.01.7}
\hat\mu_j(s, n, 0) = \begin{cases}\exp\{-\sigma_j s^2+i\mathfrak{m}_j s\}, &\text{\ if\ }
\ \ s\in \mathbb{R}, \ n=0,\\
\varkappa_j\exp\{-\sigma_j' s^2+i\mathfrak{m}_j's\}, &\text{\ if\ }\ \ s\in \mathbb{R}, 
\   n=1 
\\
\end{cases}
\end{equation}
  and either
$$
0<\sigma_j'<\sigma_j, \quad 
0<|\varkappa_j|\le\sqrt\frac{\sigma_j'}{\sigma_j}
\exp\left\{-\frac{(\mathfrak{m}_j-\mathfrak{m}_j')^2}
{4(\sigma_j-\sigma_j')}\right\} 
$$
 or
 $$
\sigma_j=\sigma_j', \quad\mathfrak{m}_j=\mathfrak{m}_j', 
\quad  |\varkappa_j|\le 1, \quad j=1, 2.
$$
Substituting (\ref{03.01.7})  into equation (\ref{03.01.5}), we get 
\begin{equation}\label{03.01.8}
\sigma_1+a\sigma_2=0, \quad \sigma'_1+a\sigma'_2=0, 
\quad \mathfrak{m}_1+a\mathfrak{m}_2=0, \quad \mathfrak{m}'_1+a\mathfrak{m}'_2=0.
\end{equation}
It follows from (\ref{03.01.8}) that either $\sigma_1> 0$, $\sigma_2> 0$ or
$\sigma_1=\sigma_2=0$.
 In view of (\ref{03.01.7}), we have
\begin{equation}\label{10.01.2}
\hat\mu_j(s, 0, 0)=\exp\{-\sigma_j s^2+i\mathfrak{m}_j s\}, \quad s\in \mathbb{R}, \ j=1, 2.
\end{equation}

Assume $\sigma_1> 0$, $\sigma_2> 0$. 
Using Lemma \ref{le4} and (\ref{10.01.2}) and reasoning as
in the proof of Lemma 3.3 in \cite{Rima}, we make sure that
 for each fixed $n\in\mathbb{Z}(2)$,  $h\in H$,
 the   functions $\hat\mu_j(s, n, h)$ can be extended to the complex 
 plane $\mathbb{C}$ as entire functions in   $s$, equation   
 (\ref{03.01.4})  holds for all   $s_j\in \mathbb{C}$, 
 $n_j\in\mathbb{Z}(2)$, $h_j\in H$ and the functions $\hat\mu_j(s, n, h)$ 
 do not vanish for all $s\in \mathbb{C}$, $n\in\mathbb{Z}(2)$, $h\in H$.   
   
3. By Lemma \ref{le4}, it follows from  (\ref{14.03.1}) and  (\ref{10.01.2}) that the functions $\hat\mu_j(s, n, h)$  for each fixed $n\in\mathbb{Z}(2)$,
$h\in H$, are entire functions in    $s$  of the order at most   2.  
Taking into account that the functions $\hat\mu_j(s, n, h)$ do not vanish,
we can apply 
the Hadamard theorem on the representation of an entire function of finite order and     obtain that the characteristic functions $\hat\mu_j(s, n, h)$ are of the form
\begin{equation}\label{03.01.12}
\hat\mu_j(s, n, h) = \begin{cases}\exp\{a_{j}(0, h)s^2+b_{j}(0, h)s+c_{j}(0, h)\}, 
&\text{\ if\ }
\ \ s\in \mathbb{R}, \ n=0, \ h\in H,\\
\exp\{a_{j}(1, h)s^2+b_{j}(1, h)s+c_{j}(1, h)\},  
&\text{\ if\ }\ \ s\in \mathbb{R}, 
\   n=1, \ h\in H,
\\
\end{cases}
\end{equation}
where $a_{j}(n, h)$, $b_{j}(n, h)$, $c_{j}(n, h)$ are some complex valued functions in $n$ and $h$. 
We will prove that in fact
\begin{equation}\label{04.01.9}
\hat\mu_j(s, n, h) = \begin{cases}\exp\{-\sigma_j s^2+i\mathfrak{m}_j s+ c_j(0, h)\}, 
&\text{\ if\ }
\ \ s\in \mathbb{R}, \ n=0, \ h\in H,\\
\exp\{-\sigma_j' s^2+i\mathfrak{m}_j's+c_j(1, h)\}, 
&\text{\ if\ }\ \ s\in \mathbb{R}, 
\   n=1, \ h\in H,
\\
\end{cases}
\end{equation}
where either
$
0<\sigma_j'<\sigma_j 
$ 
or
$
0<\sigma_j=\sigma_j'$, $\mathfrak{m}_j=\mathfrak{m}_j'$, $j=1, 2$.

First we will prove that 
$$a_{j}(0, h)=-\sigma_j, \ \ a_{j}(1, h)=-\sigma'_j, \ \  b_{j}(0, h)=i\mathfrak{m}_j,  \ \ b_{j}(1, h)=i\mathfrak{m}'_j, \quad h\in H, \ \ j=1, 2.$$

By the condition of the theorem,
$G$ is a finite Abelian group containing no elements of order  $2$. Hence 
$H^{(2)}=H$. Substitute    representations (\ref{03.01.12}) for the characteristic functions 
$\hat\mu_j(s, n, h)$   into equation (\ref{03.01.4}). Put
$n_1=n_2=0$, $h_1=h_2=h$ in the obtained equation.  Since $H^{(2)}=H$, we have
\begin{multline}
\label{04.01.1}
a_{1}(0, h)(s_1+s_2)^2+b_{1}(0, h)(s_1+s_2)+c_{1}(0, h)+
a_{2}(0, 0)(s_1+as_2)^2+b_{2}(0, 0)(s_1+as_2)\\+c_{2}(0, 0)=  
a_{1}(0, 0)(s_1-s_2)^2+b_{1}(0, 0)(s_1-s_2)+c_{1}(0, 0)+
a_{2}(0, h)(s_1-as_2)^2\\+b_{2}(0, h)(s_1-as_2)+c_{2}(0, h)
+2\pi ik_0(h), \quad  s_j\in \mathbb{R},  
 \ h\in H,
\end{multline}  
where $k_0(h)$ is an integer valued function in $h$.  

Equating  the coefficients  of
$s_1^2$ and  $s_2^2$ on each  side  of equation (\ref{04.01.1}), we get
\begin{equation}\label{04.01.2}
a_{1}(0, h)+a_{2}(0, 0)=a_{1}(0, 0)+a_{2}(0, h), \quad 
a_{1}(0, h)+a^2a_{2}(0, 0)=a_{1}(0, 0)+a^2a_{2}(0, h), \ h\in H.
\end{equation}
It follows from (\ref{03.01.7}) that
\begin{equation}\label{04.01.3}
a_{1}(0, 0)=-\sigma_1, \quad a_{2}(0, 0)=-\sigma_2.
\end{equation}
Taking into account (\ref{03.01.8}), we obtain from 
(\ref{04.01.2}) and (\ref{04.01.3}) that
\begin{equation}\label{04.01.4}
a_{1}(0, h)=-\sigma_1, \quad a_{2}(0, h)=-\sigma_2, \quad h\in H.
\end{equation}

Equating  the coefficients  of
$s_1$ and  $s_2$ on each  side  of equation (\ref{04.01.1}), we receive
\begin{equation}\label{04.01.5}
b_{1}(0, h)+b_{2}(0, 0)=b_{1}(0, 0)+b_{2}(0, h), \quad 
b_{1}(0, h)+ab_{2}(0, 0)=-b_{1}(0, 0)-ab_{2}(0, h).
\end{equation}
It follows from (\ref{03.01.7}) that
\begin{equation}\label{04.01.6}
b_{1}(0, 0)=i\mathfrak{m}_1, \quad b_{2}(0, 0)=i\mathfrak{m}_2.
\end{equation}
Taking into account (\ref{03.01.8}), we obtain from 
(\ref{04.01.5}) and (\ref{04.01.6}) that
\begin{equation}\label{04.01.7}
b_{1}(0, h)=i\mathfrak{m}_1, \quad b_{2}(0, h)=i\mathfrak{m}_2, \quad h\in H.
\end{equation}

Substitute representations (\ref{03.01.12}) for the characteristic functions 
$\hat\mu_j(s, n, h)$   into equation (\ref{03.01.4}). Setting now
$n_1=1$, $n_2=0$, $h_1=h_2=h$ in the obtained equation and taking into account that
$H^{(2)}=H$, we get
\begin{multline}
\label{04.01.8}
a_{1}(1, h)(s_1+s_2)^2+b_{1}(1, h)(s_1+s_2)+c_{1}(1, h)+
a_{2}(1, 0)(s_1+as_2)^2+b_{2}(1, 0)(s_1+as_2)\\+c_{2}(1, 0)=  
a_{1}(1, 0)(s_1-s_2)^2+b_{1}(1, 0)(s_1-s_2)+c_{1}(1, 0)+
a_{2}(1, h)(s_1-as_2)^2\\+b_{2}(1, h)(s_1-as_2)+c_{2}(1, h)
+2\pi ik_1(h), \quad  s_j\in \mathbb{R},  
 \ h\in H,
\end{multline}  
where $k_1(h)$ is an integer valued function in $h$. Arguing as in the case when $n_1=n_2=0$, we get from equation (\ref{04.01.8})
$$
a_{1}(1, h)=-\sigma'_1, \quad a_{2}(1, h)=-\sigma'_2, \quad h\in H.
$$
and
$$
b_{1}(1, h)=i\mathfrak{m}'_1, \quad b_{2}(1, h)=i\mathfrak{m}'_2, \quad h\in H.
$$
In view of (\ref{04.01.4}) and (\ref{04.01.7}), as a result representations (\ref{03.01.12}) take the form (\ref{04.01.9}).

4. Suppose $
0<\sigma_j'<\sigma_j 
$, $j=1, 2$.  Put $p=(0, 1, 0)\in X$, i.e., $p$ is the elements of order 2 of the group 
$X$. We recall that $\gamma_{\sigma, \mathfrak{m}}$ 
is the Gaussian distribution on the group
$\mathbb{R}$ with the characteristic function of the form
(\ref{26.01.1}). Consider the signed measures   
$$
\lambda_j=\frac{1}{2}\left(\gamma_{\sigma_j, \mathfrak{m}_j}+
\gamma_{\sigma'_j, \mathfrak{m}'_j}\right)+\frac{1}{2}\left(\gamma_{\sigma_j, \mathfrak{m}_j}-
\gamma_{\sigma'_j, \mathfrak{m}'_j}\right)*E_p, \quad j=1, 2,
$$
on the subgroup   
$\mathbb{R}\times \mathbb{Z}(2)$.
Then the characteristic functions $\hat\lambda_j(s, n, h)$ are of the form
\begin{equation}\label{17.01.1}
\hat\lambda_j(s, n, h) = \begin{cases}\exp\{-\sigma_j s^2+i\mathfrak{m}_j s\}, 
&\text{\ if\ }
\ \ s\in \mathbb{R}, \ n=0, \ h\in H,\\
\exp\{-\sigma_j' s^2+i\mathfrak{m}_j's\}, 
&\text{\ if\ }\ \ s\in \mathbb{R}, 
\   n=1, \ h\in H, \quad j=1, 2.
\\
\end{cases}
\end{equation} 

Denote by $\tau_j$  the distributions on the group $X$ with the characteristic 
functions 
\begin{equation}\label{04.01.11}
\hat\tau_j(s, n, h) = \begin{cases}\exp\{c_j(0, h)\}, 
&\text{\ if\ }
\ \ s\in \mathbb{R}, \ n=0, \ h\in H,\\
\exp\{c_j(1, h)\}, &\text{\ if\ }\ \ s\in \mathbb{R}, 
\   n=1, \ h\in H, \quad j=1, 2.
\\
\end{cases}
\end{equation}
It is obvious that 
\begin{equation}\label{18.01.20}
\hat\tau_j(s, n, h)=\hat\mu_j(0, n, h) \quad s\in \mathbb{R}, \ n\in\mathbb{Z}(2),  \  h\in H,
\end{equation}
and the distributions $\tau_j$ are supported 
in the subgroup $\mathbb{Z}(2)\times G$. It follows from (\ref{04.01.9}),(\ref{17.01.1}), and (\ref{04.01.11}) that 
$$
\hat\mu_j(s, n, h)=\hat\lambda_j(s, n, h)\hat\tau_j(s, n, h), \quad s\in \mathbb{R}, \ n\in\mathbb{Z}(2),  \ h\in H, \   j=1, 2.
$$
Hence
$$
\mu_j=\lambda_j*\tau_j, \quad j=1, 2.
$$

 The distributions $\tau_j$ can be written in the form 
$$
\tau_j=\sum_{g_i\in G} a^{(j)}_iE_{g_i}+\sum_{g_i\in G} b^{(j)}_iE_{g_i+p}, 
$$
where $a^{(j)}_i\ge 0$,   $b^{(j)}_i\ge 0$,  $j=1, 2$.
Since $\mu_j$ are distributions, by Lemma \ref{le17.2}, the inequalities 
\begin{equation}\label{17.01.2}
\left|\frac{a^{(j)}_i-b^{(j)}_i}{a^{(j)}_i+b^{(j)}_i}\right|\le\sqrt\frac{\sigma_j'}{\sigma_j}\exp\left\{-\frac{(\mathfrak{m}_j-\mathfrak{m}_j')^2}
{4(\sigma_j-\sigma_j')}\right\},\quad j=1, 2,
\end{equation}
hold  for all $a^{(j)}_i$, $b^{(j)}_i$  such that 
 $a^{(j)}_i+b^{(j)}_i>0$, $j=1, 2$.

Denote by $\gamma_j$  the distributions of the class $\Theta$ on the subgroup   
$\mathbb{R}\times \mathbb{Z}(2)$ of the group $X$ with the characteristic functions 
\begin{equation}\label{04.01.10}
\hat\gamma_j(s, n, h) = \begin{cases}\exp\{-\sigma_j s^2+i\mathfrak{m}_j s\}, 
&\text{\ if\ }
\ \ s\in \mathbb{R}, \ n=0, \ h\in H,\\
\rho_j\exp\{-\sigma_j' s^2+i\mathfrak{m}_j's\}, &\text{\ if\ }\ \ s\in \mathbb{R}, 
\   n=1, \ h\in H,
\\
\end{cases}
\end{equation}
where
\begin{equation}\label{17.01.3}
\rho_j=\sqrt\frac{\sigma_j'}{\sigma_j}
\exp\left\{-\frac{(\mathfrak{m}_j-\mathfrak{m}_j')^2}{
4(\sigma_j-\sigma_j')}\right\}, \quad j=1, 2. 
\end{equation}

Consider on the subgroup $\mathbb{Z}(2)$ the signed  measures
$$
\pi_j=\frac{\rho_j+1}{2\rho_j}E_0+\frac{\rho_j-1}{2\rho_j}E_p, 
\quad j=1, 2.
$$
Then
\begin{equation}\label{04.01.12}
\hat\pi_j(s, n, h) = \begin{cases}1, 
&\text{\ if\ }
\ \ s\in \mathbb{R}, \ n=0, \ h\in H,\\
\rho_j^{-1}, &\text{\ if\ }\ \ s\in \mathbb{R}, 
\   n=1, \ h\in H, \quad j=1, 2.
\\
\end{cases}
\end{equation}
It follows from (\ref{04.01.9}), (\ref{04.01.11}), (\ref{04.01.10}), and (\ref{04.01.12}) that
$$\hat\mu_j(s, n, h)=\hat\gamma_j(s, n, h)\hat\tau_j(s, n, h)\hat\pi_j(s, n, h), 
\quad s\in \mathbb{R}, \ n\in\mathbb{Z}(2),  \ h\in H, \   j=1, 2.
$$
This implies that $\mu_j=\gamma_j*\tau_j*\pi_j$, $j=1, 2$. 

5. Put $\omega_j=\tau_j*\pi_j$. Then
$\omega_j$ are signed measures on the subgroup  $\mathbb{Z}(2)\times G$.
We will verify that in fact $\omega_j$  are distributions. We have
\begin{multline*}
\omega_j=\tau_j*\pi_j=\left(\sum_{g_i\in G} a^{(j)}_iE_{g_i}+\sum_{g_i\in G} b^{(j)}_iE_{g_i+p}\right)*\left(\frac{\rho_j+1}{2\rho_j}E_0+\frac{\rho_j-1}{2\rho_j}E_p\right)\\=\sum_{g_i\in G} \frac{a^{(j)}_i(\rho_j+1)}{2\rho_j}E_{g_i}+
\sum_{g_i\in G} \frac{b^{(j)}_i(\rho_j+1)}{2\rho_j}E_{g_i+p}\\+
\sum_{g_i\in G} \frac{a^{(j)}_i(\rho_j-1)}{2\rho_j}E_{g_i+p}+
\sum_{g_i\in G} \frac{b^{(j)}_i(\rho_j-1)}{2\rho_j}E_{g_i}, \quad j=1, 2.
\end{multline*}
This implies that 
\begin{equation}\label{18.01.1}
\omega_j(\{g_i\})=\frac{a^{(j)}_i(\rho_j+1)}{2\rho_j}+\frac{b^{(j)}_i(\rho_j-1)}{2\rho_j}=\frac{1}{2\rho_j}\left(\left(a^{(j)}_i+b^{(j)}_i\right)
\rho_j+\left(a^{(j)}_i-b^{(j)}_i\right)\right), \quad j=1, 2,
\end{equation}
and
\begin{equation}\label{18.01.2}
\omega_j(\{g_i+p\})=\frac{b^{(j)}_i(\rho_j+1)}{2\rho_j}+\frac{a^{(j)}_i(\rho_j-1)}{2\rho_j}=\frac{1}{2\rho_j}\left(\left(a^{(j)}_i+b^{(j)}_i\right)
\rho_j+\left(b^{(j)}_i-a^{(j)}_i\right)\right), \quad j=1, 2.
\end{equation}
From (\ref{17.01.2}) and  (\ref{17.01.3}) we get
\begin{equation}\label{25.01.10}
\left|\frac{a^{(j)}_i-b^{(j)}_i}{a^{(j)}_i+b^{(j)}_i}\right|\le\rho_j 
\end{equation}
for all $a^{(j)}_i$, $b^{(j)}_i$ such that $a^{(j)}_i+b^{(j)}_i>0$, $j=1, 2$. 
It follows from (\ref{18.01.1})--(\ref{25.01.10}) that 
$\omega_j(\{g_i\})\ge 0$ and $\omega_j(\{g_i+p\})\ge 0$
for all $g_i\in G$, $j=1, 2$. 

6. To complete the proof when  $\sigma_1> 0$, $\sigma_2> 0$ in (\ref{03.01.7})
it remains to prove that  either 
$\omega_1=\omega_2*\vartheta_2$ or   $\omega_2=\omega_1*\vartheta_1$, 
 where  $\vartheta_j\in \mathrm{M}^1(\mathbb{Z}(2))$.
 Substitute $s_1=s_2=0$ in 
equation (\ref{03.01.4}). In view of (\ref{18.01.20}),
 we obtain
\begin{multline}
\label{18.01.10}
\hat\tau_1(0, n_1+n_2, h_1+h_2)\hat\tau_2(0, n_1+n_2, h_1-h_2) 
\\=
\hat\tau_1(0, n_1+n_2, h_1-h_2)\hat\tau_2(0, n_1+n_2, h_1+h_2), 
\quad  n_j\in\mathbb{Z}(2), \ h_j\in H.
\end{multline} 
Taking into account Lemma \ref{le1}, it follows from
equation (\ref{18.01.10})  and  Lemma \ref{le5}  applied 
to the group $\mathbb{Z}(2)\times G$  that
either $\tau_1=\tau_2*\delta_2$ or  $\tau_2=\tau_1*\delta_1$, where 
$\delta_j\in{\rm M}^1(\mathbb{Z}(2))$, $j=1, 2$. Assume for definiteness that
$\tau_1=\tau_2*\delta_2$. 

We note that the set of all signed measures $\pi$ on the group 
$\mathbb{Z}(2)$ with nonvanishing characteristic functions and such
that $\hat\pi(0)=1$ forms an Abelian group with respect to the convolution 
with the identity element $E_0$.
It means that for each such $\pi$ there is a signed measure $\pi^{-1}$ on the 
group $\mathbb{Z}(2)$ such that $\pi*\pi^{-1}=E_0$.
Moreover, either  $\pi$ or $\pi^{-1}$ is a distribution.

 We have 
$\omega_1=\tau_1*\pi_1=\tau_2*\delta_2*\pi_1*\pi_2*\pi_2^{-1}=\omega_2*\delta_2*\pi_1*\pi_2^{-1}$. Put 
$\vartheta=\delta_2*\pi_1*\pi_2^{-1}$. Then $\vartheta$ is a signed measure on
the subgroup $\mathbb{Z}(2)$ and  either $\vartheta$ or $\vartheta^{-1}$
is a distribution. We have $\omega_1=\omega_2*\vartheta$. 
Put $\vartheta_2=\vartheta$ if $\vartheta$ 
is a distribution and put $\vartheta_1=\vartheta$ if $\vartheta^{-1}$ 
is a distribution.
Thus, we proved the theorem in the case when $0<\sigma_j'<\sigma_j$, $j=1, 2$.

Suppose $0<\sigma_j=\sigma_j'$, $\mathfrak{m}_j=\mathfrak{m}_j'$, $j=1, 2$. Then
$\mu_j=\gamma_j*\tau_j$, where $\gamma_j$ are Gaussian distributions on 
the subgroup 
$\mathbb{R}$ with the characteristic functions
$$\gamma_j(s, n, h)=\exp\{-\sigma_j s^2+i\mathfrak{m}_j s\}, \quad
s\in \mathbb{R}, \ n\in\mathbb{Z}(2), \   h\in H, \ j=1, 2,$$ 
and $\tau_j$, $j=1, 2$, are  given by (\ref{04.01.11}). 

Thus, we have completely proved the theorem in the case when $\sigma_1> 0$, 
$\sigma_2> 0$ in (\ref{03.01.7}).
 
7. Assume   $\sigma_1=\sigma_2=0$ in (\ref{03.01.7}). Then 
  (\ref{10.01.2}) takes the form
$$
\hat\mu_j(s, 0, 0)=\exp\{i\mathfrak{m}_j s\}, \quad s\in \mathbb{R}, \ j=1, 2.
$$
Put $\nu_j=\mu_j*E_{-\mathfrak{m}_j}$. We have $\hat\nu_j(s, 0, 0)=1$ for all 
$s\in\mathbb{R}$, $j=1, 2$.   By Lemma \ref{le6}, this implies  that the distributions 
$\nu_j$ are supported in the subgroup $A(X, \mathbb{R})=\mathbb{Z}(2)\times G$.
Let $\eta_j$ be independent random variables with values in the group $X$ and distributions  $\nu_j$.  In view of (\ref{03.01.8}),  
$\mathfrak{m}_1+a\mathfrak{m}_2=0$.
Since
$\hat\nu_j(y)=\hat\mu_j(y)(-\mathfrak{m}_j, y)$, $y\in Y$, $j=1, 2$,  
the characteristic functions $\hat\nu_j(y)$ satisfy equation (\ref{03.01.1}).
Hence by  Lemma \ref{le1}, the conditional distribution of the linear form
$M_2=\eta_1 + \alpha\eta_2$ given $M_1=\eta_1 + \eta_2$ is   symmetric. 
Note that the restriction of $\alpha$ to the subgroup 
$\mathbb{Z}(2)\times G$ coincides with $-I$. Therefore, if
we consider  $\eta_j$ as independent  
random variables with values in the subgroup  $\mathbb{Z}(2)\times G$, then the conditional distribution of the 
linear form $M_2 = \eta_1 -\eta_2$ given $M_1 = \eta_1 + \eta_2$  is symmetric.
The statement of the theorem follows from Lemma \ref{le5} applied 
to the group $\mathbb{Z}(2)\times G$. 

8. It remains to consider the case when $a=-1$. Then $\alpha=-I$ and the
statement of the theorem also follows   from Lemma \ref{le5}. 
\end{proof}

\section { Remarks and Discussions}
    
\begin{remark}\label{re1} We will verify that Theorem 
\ref{th1} 
can not be strengthened by narrowing the
class of distributions which are characterized by the symmetry of the conditional
distribution of one linear form   given another.
We retain the notation used in the formulation of Theorem 
\ref{th1}. 
 
 Suppose  $a \ne-1$. Let $\gamma_j$ be distributions of the 
class $\Theta$ on the group   
$\mathbb{R}\times \mathbb{Z}(2)$  with the nonvanishing characteristic 
functions of the form 
$$
\hat\gamma_j(s, n) = \begin{cases}\exp\{-\sigma_j s^2+i\mathfrak{m}_j s\}, &\text{\ if\ }
\ \ s\in \mathbb{R}, \ n=0,\\
\varkappa_j\exp\{-\sigma_j' s^2+i\mathfrak{m}_j's\}, &\text{\ if\ }\ \ s\in \mathbb{R}, 
\   n=1, \ \ j=1, 2. 
\\
\end{cases}
$$
Assume that equalities (\ref{03.01.8}) 
are fulfilled. Then the characteristic functions $\hat\gamma_j(s, n)$ satisfy the equation 
\begin{multline*}
\hat\gamma_1(s_1+s_2, n_1+n_2)\hat\gamma_2(s_1+a s_2, n_1+n_2) 
\\=
\hat\gamma_1(s_1-s_2, n_1+n_2)\hat\gamma_2(s_1-a s_2, n_1+n_2), \quad s_j\in \mathbb{R}, \  n_j\in\mathbb{Z}(2).
\end{multline*}
Obviously, if we consider the group   
$\mathbb{R}\times \mathbb{Z}(2)$ as a subgroup of the group $X$ and $\gamma_j$ as distributions on  $X$, then the characteristic functions 
$\hat\gamma_j(s, n, h)$ satisfy equation (\ref{03.01.2}). 

Put $K={\rm Ker}(I+\alpha_G)$. Then the restriction of the automorphism   
  $\alpha_{G}$  to the subgroup $K$ coincides with $-I$. Let  $\omega_j$ be distributions on  
the subgroup $\mathbb{Z}(2)\times K$ with the nonvanishing characteristic functions
and such that either 
$\omega_1=\omega_2*\vartheta_2$ or   $\omega_2=\omega_1*\vartheta_1$, 
 where  $\vartheta_j\in \mathrm{M}^1(\mathbb{Z}(2))$, $j=1, 2$. 
It follows from Lemmas \ref{le1} and \ref{le5} that 
if we consider $\omega_j$ as distributions on the group $X$, then the characteristic functions $\hat\omega_j(s, n, h)$ satisfy equation (\ref{03.01.2}).

Let $x_j=(t_j, m_j, g_j)$  be elements of the group $X$
 such that 
\begin{equation}\label{08.01.10}
x_1+\alpha x_2=0.
\end{equation}
It follows from (\ref{08.01.10}) that the 
characteristic functions $((t_j, m_j, g_j),(s, n, h))$ satisfy  equation (\ref{03.01.2}).

Consider $\gamma_j$ and $\omega_j$ as distributions on the group
$X$ and put $\mu_j=\gamma_j*\omega_j*E_{x_j}$, $j=1, 2$. 
Then 
$$
\hat\mu_j(s, n, h)=\hat\gamma_j(s, n, h)
\hat\omega_j(s, n, h)((t_j, m_j, g_j),(s, n, h)), \quad s\in \mathbb{R}, \  n\in\mathbb{Z}(2), \ h\in H, \ \ j=1, 2.
$$ 
From the above it follows that  the 
characteristic functions $\hat\mu_j(s, n, h)$ satisfy  equation (\ref{03.01.2}). 
Let $\xi_j$, $j=1, 2$,  be
independent random variables with values in the group
$X$ and distributions $\mu_j$. Since the characteristic functions 
$\hat\mu_j(s, n, h)$ satisfy equation  (\ref{03.01.2}),   
Lemma  \ref{le1} implies that  the conditional  distribution of the 
linear form $L_2 = \xi_1 + \alpha\xi_2$ given $L_1 = \xi_1 + \xi_2$ is symmetric.

Suppose $a=-1$.   Let  $\omega_j$ be distributions on  the subgroup
  $\mathbb{R}\times\mathbb{Z}(2)\times K$ with the nonvanishing 
  characteristic functions and such that either 
$\omega_1=\omega_2*\vartheta_2$ or   $\omega_2=\omega_1*\vartheta_1$, 
 where  $\vartheta_j\in \mathrm{M}^1(\mathbb{Z}(2))$, $j=1, 2$. 
Let $x_j$  be 
elements of the group $X$ satisfying (\ref{08.01.10}).   
Put $\mu_j=\omega_j*E_{x_j}$, $j=1, 2$.  It follows from 
Lemmas \ref{le1} and \ref{le5} that the characteristic functions 
 $\hat\mu_j(s, n, h)$ satisfy equation (\ref{03.01.2}). By Lemma  
 \ref{le1},  the conditional  
distribution of the linear form $L_2 = \xi_1 + \alpha\xi_2$ 
given $L_1 = \xi_1 + \xi_2$ is symmetric.
 \end{remark}  
 Compare Theorem \ref{th1} with the corresponding result
for the group $X=\mathbb{R}\times G$, where $G$ is a finite Abelian 
 group containing no elements of order 2 (see \cite[Theorem 3.1]{Rima}).
 
\begin{theorem} Let $X=\mathbb{R}\times G$, where $G$ is a finite Abelian group 
containing no elements of order  $2$.
Let  $\alpha=(a, \alpha_{G})$ be a topological automorphism of the group
  $X$.  Set $K={\rm Ker}(I+\alpha_{G})$.
Let $\xi_j$, $j=1, 2$, be
independent random variables with values in   $X$ and distributions
$\mu_j$ with nonvanishing characteristic functions. Assume that the conditional  distribution of the linear form $L_2 = \xi_1 + \alpha\xi_2$ given 
$L_1 = \xi_1 + \xi_2$ is symmetric. Then the following statements are true. 
\renewcommand{\labelenumi}{\rm{\Roman{enumi}.}}
\begin{enumerate}
  
\item	

 If  $a \ne-1$, then $\mu_j=\gamma_j*\omega*E_{x_j}$, where $\gamma_j$ are Gaussian distributions on the subgroup  $\mathbb{R}$, $\omega$ is a distribution on
 the subgroup $K$, $x_j\in X$.  

\item

 If  $a=-1$, then $\mu_j=\omega*E_{x_j}$, where  $\omega$ is a distribution on the subgroup  $\mathbb{R}\times K$, $x_j\in X$, $j=1, 2$.   
\end{enumerate}
\end{theorem}  

We see that the presence of even one element of order 2 in a group $X$ makes the description of distributions which are characterized by the symmetry of 
the conditional distribution of one linear 
form of independent random variables given another  more complicated.
 
\medskip
 
Assume that $X=\mathbb{R}\times \mathbb{Z}(2)\times G$, where $G$ is a finite Abelian 
 group. Let $\mu=\gamma*\omega$, 
 where $\gamma$ is a distribution of the class 
 $\Theta$ on the subgroup
 $\mathbb{R}\times\mathbb{Z}(2)$  with  the nonvanishing characteristic 
 function, and $\omega$   is a 
distribution on the subgroup $\mathbb{Z}(2)\times G$ with the nonvanishing characteristic function.
In connection with Theorem \ref{th1} at the end of the article, we discuss the question of the uniqueness of such a representation for $\mu$.
We need the following easily verified assertion.
\begin{lemma}\label{lenew11}
Assume that $X=\mathbb{R}\times \mathbb{Z}(2)\times G$, 
where $G$ is a finite Abelian 
 group.  Consider the distribution $\gamma$ of the class 
 $\Theta$ on the subgroup
 $\mathbb{R}\times\mathbb{Z}(2)$  with the characteristic 
 function of the form 
\begin{equation}\label{04.03.5}
\hat\gamma(s, n, h)=\begin{cases}\exp\{-\sigma  s^2+i\mathfrak{m}  s\}, &\text{\ if\ }\ 
\ s\in \mathbb{R},   
\ n=0, \   h\in H,\\
\varkappa\exp\{-\sigma' s^2+i\mathfrak{m}'s\}, &\text{\ if\ }\ \ s\in \mathbb{R}, 
\    n=1, \   h\in H,
\\
\end{cases}
\end{equation} 
such that inequalities $(\ref{08.01.8})$  are satisfied.  Let $\omega$
 be a distribution on the subgroup $\mathbb{Z}(2)\times G$ with the nonvanishing characteristic function. Put $\mu=\gamma*\omega$. Assume 
 $\mu=\gamma_1*\omega_1$, where $\gamma_1\in\Theta$  and
 $\omega_1\in {\rm M}^1(\mathbb{Z}(2)\times G)$. Then $\gamma_1=\gamma*\pi$
 and  $\omega_1=\omega*\pi^{-1}$, where $\pi$ is a signed measure  
 on the subgroup $\mathbb{Z}(2)$. 
\end{lemma}
\begin{proof} Assume
\begin{equation}\label{04.03.3}
\hat\gamma_1(s, n, h)=\begin{cases}\exp\{-\sigma_1 s^2+i\mathfrak{m}_1 s\}, &\text{\ if\ }\ 
\ s\in \mathbb{R}, 
\ n=0, \   h\in H,\\
\varkappa_1\exp\{-\sigma'_1 s^2+i\mathfrak{m}'_1s\}, &\text{\ if\ }\ \ s\in \mathbb{R}, 
\    n=1, \   h\in H.
\\
\end{cases}
\end{equation}
We have 
\begin{equation}\label{04.03.1}
\hat\gamma(s, n, h)\hat\omega(s, n, h)=\hat\gamma_1(s, n, h)\hat\omega_1(s, n, h), 
\quad
 s\in \mathbb{R}, \ n\in \mathbb{Z}(2), \ h\in H.
\end{equation}
Note that the characteristic functions $\hat\omega(s, n, h)$ and 
$\hat\omega_1(s, n, h)$ do not depend on $s$.
Substituting $n=0$, $h=0$ into  (\ref{04.03.1}) and taking into 
account (\ref{04.03.5}) 
and (\ref{04.03.3}), we find
\begin{equation}\label{04.03.2}
\sigma=\sigma_1, \quad    \mathfrak{m}=\mathfrak{m}_1.
\end{equation} 
Substituting $n=1$, $h=0$ into  (\ref{04.03.1}) and taking into 
account (\ref{04.03.5}) and (\ref{04.03.3}), we obtain
\begin{equation}\label{04.03.4}
\sigma'=\sigma'_1, \quad   \mathfrak{m}'=\mathfrak{m}'_1.
\end{equation} 
Denote by $\pi$ the signed measure on the subgroup $\mathbb{Z}(2)$ 
with the characteristic function
$$
\hat\pi(s, n, h) = \begin{cases}1, 
&\text{\ if\ }\
\ s\in \mathbb{R}, 
\ n=0, \   h\in H, \\
\frac{\varkappa_1}{\varkappa}, &\text{\ if\ }\ \   
     s\in \mathbb{R}, 
\ n=0, \   h\in H.
\\
\end{cases}
$$
It follows from  (\ref{04.03.5}), (\ref{04.03.3}),  (\ref{04.03.2}) and (\ref{04.03.4})
that $\hat\gamma_1(s, n, h)=\hat\gamma(s, n, h)\hat\pi(s, n, h)$ for all 
$s\in \mathbb{R}$, $n\in \mathbb{Z}(2)$, $h\in H$. Taking this into account,
(\ref{04.03.1}) implies that 
$\hat\omega_1(s, n, h)=\hat\omega(s, n, h)\hat\pi^{-1}(s, n, h)$ 
for all $s\in \mathbb{R}$,  $n\in \mathbb{Z}(2)$, $h\in H$. 
Hence $\gamma_1=\gamma*\pi$
 and  $\omega_1=\omega*\pi^{-1}$. \end{proof}

\begin{proposition}\label{pr1}
 Assume that $X=\mathbb{R}\times \mathbb{Z}(2)\times G$, where $G$ is a finite Abelian 
 group.  Let $p=(0, 1, 0)$ be the element of order $2$ of the  subgroup
 $\mathbb{Z}(2)$. Let 
 $\gamma$ be a distribution of the class 
 $\Theta$ on the subgroup
 $\mathbb{R}\times\mathbb{Z}(2)$  with the characteristic 
 function of the form $\hat\gamma(s, n)=\phi(s, n)$, where  
 the function   $\phi(s, n)$ 
is represented in the form 
$(\ref{21.01.1})$, and inequalities 
 $(\ref{08.01.8})$  are satisfied.
 Consider a 
distribution $\omega$ on the subgroup $\mathbb{Z}(2)\times G$ of the form
\begin{equation}\label{21.01.10}
\omega=\sum_{g_i\in G} a_iE_{g_i}+\sum_{g_i\in G} b_iE_{g_i+p}, 
\end{equation} 
where $a_i\ge 0$,  $b_i\ge 0$, and suppose that the characteristic function 
 $\hat\omega(n, h)$
does not vanish. 
Put $\mu=\gamma*\omega$. Assume $\mu=\gamma_1*\omega_1$, where $\gamma_1\in\Theta$  and
 $\omega_1\in {\rm M}^1(\mathbb{Z}(2)\times G)$.  This implies that
  $\gamma_1=\gamma*E_m$, $\omega_1=\omega*E_{m}$, where $m\in  \mathbb{Z}(2)$, 
  if and only if 
 \begin{equation}\label{21.01.12}
|\varkappa|=\sqrt\frac{\sigma'}{\sigma}\exp\left\{-\frac{(\mathfrak{m}-\mathfrak{m}')^2}
{4(\sigma-\sigma')}\right\}
\end{equation} 
and there is an element $g_i\in G$ such that either $a_i=0$ 
and $b_i>0$ or $a_i>0$ and $b_i=0$.
\end{proposition}

\begin{proof} Sufficiency. Assume that (\ref{21.01.12}) holds. 
By Lemma  \ref{lenew11},   $\gamma_1=\gamma*\pi$ and $\omega_1=\omega*\pi^{-1}$, where $\pi$ is a signed measure on $\mathbb{Z}(2)$.
Taking into account (\ref{21.01.12}), it follows from  Lemma \ref{le19.01.1} 
  that in fact 
$\pi$ is a distribution. Suppose that $\pi$ is not a degenerate distribution. 
Then $\pi^{-1}$ is a singed measure but not a distribution. We can write  
the characteristic function 
 $\hat\pi^{-1}(s, n, h)$ in the form
  \begin{equation}\label{22.01.1}
\hat\pi^{-1}(s, n, h) = \begin{cases}1, 
&\text{\ if\ }
\ \ s\in \mathbb{R}, \ n=0, \ h\in H,\\
c^{-1}, &\text{\ if\ }\ \ s\in \mathbb{R}, 
\   n=1, \ h\in H,
\\
\end{cases}
\end{equation} 
where $0<|c|<1$.  It follows from (\ref{22.01.1}) that 
\begin{equation}\label{20.01.11}
\pi^{-1}=\frac{c+1}{2c}E_0+\frac{c-1}{2c}E_p. 
\end{equation} 
We can assume without loss of generality that $c>0$. 
Taking into account (\ref{21.01.10}) and (\ref{20.01.11}) and repeating the reasoning in item 5 of Theorem \ref{th1} we get that
\begin{equation}\label{n18.01.1}
\omega_1(\{g_i\})=\frac{a_i(c+1)}{2c}+\frac{b_i(c-1)}{2c}=\frac{1}{2c}\left(\left(a_i+b_i\right)
c+\left(a_i-b_i\right)\right),  
\end{equation}
and
\begin{equation}\label{n18.01.2}
\omega_1(\{g_i+p\})=\frac{b_i(c+1)}{2c}+\frac{a_i(c-1)}{2c}=\frac{1}{2c}\left(\left(a_i+b_i\right)
c+\left(b_i-a_i\right)\right).
\end{equation}

Suppose that there is an element $g_i\in G$ such that  $a_i=0$ and $b_i>0$. 
It follows from
(\ref{n18.01.1}) that 
$$
\omega_1(\{g_i\})=\frac{(c-1)b_i}{2c}<0
$$
 for any $0<c<1$. Similarly, if there 
is an element $g_i\in G$ such that  $a_i>0$ and $b_i=0$,   then
(\ref{n18.01.2}) implies that $\omega_1(\{g_i+p\})<0$ for any $0<c<1$. We get that the convolution
$\omega_1=\omega*\pi^{-1}$ is not a distribution for any signed 
measure $\pi^{-1}$. The obtained contradiction 
shows that $\pi$ is a degenerate distribution.
Thus, the sufficiency is proved.

Let us prove the necessity. Assume 
$$
0<|\varkappa|<\sqrt\frac{\sigma'}{\sigma}
\exp\left\{-\frac{(\mathfrak{m}-\mathfrak{m}')^2}{
4(\sigma-\sigma')}\right\}. 
$$
 Choose a number  $\rho$   in such a way that  
$$
|\varkappa|<\rho\le\sqrt\frac{\sigma'}{\sigma}
\exp\left\{-\frac{(\mathfrak{m}-\mathfrak{m}')^2}{
4(\sigma-\sigma')}\right\}.
$$
Denote by $\pi$ the signed measure  on  the subgroup $\mathbb{Z}(2)$ with 
the characteristic function 
$$
\hat\pi(s, n, h) = \begin{cases}1, 
&\text{\ if\ }
\ \ s\in \mathbb{R}, \ n=0, \ h\in H,\\
\frac{\rho}{\varkappa}, &\text{\ if\ }\ \ s\in \mathbb{R}, 
\   n=1, \ h\in H.
\\
\end{cases}
$$
Then $\pi^{-1}$ is a nondegenerate distribution on   $\mathbb{Z}(2)$. We have
$\mu=\gamma*\omega=(\gamma*\pi)*(\omega*\pi^{-1})$. Put $\gamma_1=\gamma*\pi$ and 
$\omega_1=\omega*\pi^{-1}$.
It is obvious that
$\gamma_1\in\Theta$ and $\omega_1\in {\rm M}^1(\mathbb{Z}(2)\times G)$.
Thus, we have proved the necessity of  condition (\ref{21.01.12}).

Assume that (\ref{21.01.12}) is fulfilled and  for each $g_i\in G$ either $a_i=b_i=0$ or $a_i>0$ and $b_i>0$. Take $c$ such that $0<c<1$. Assume that a signed measure
$\pi^{-1}$ is defined by (\ref{20.01.11}). Then $\pi$ is a nondegenerate
distribution. Put $\gamma_1=\gamma*\pi$ and $\omega_1=\omega*\pi^{-1}$.
It is obvious that $\gamma_1\in\Theta$. 
It follows from  
(\ref{n18.01.1}) and (\ref{n18.01.2}) that if a number $c$ 
 is taken close enough to 1, then
$\omega_1(\{g_i\})>0$ and $\omega_1(\{g_i+p\})>0$ for all $g_i\in G$, 
where $a_i>0$ and $b_i>0$. This implies that $\omega_1\in {\rm M}^1(\mathbb{Z}(2)\times G)$. The proposition is proved completely.  
\end{proof}
     
 \medskip
 
\noindent{\bf Acknowledgements} This article was written during the author's stay at 
the Department of Mathematics University of Toronto as a Visiting Professor. 
I am very grateful to Ilia Binder for his invitation and support.
I would also like to thank the reviewer for carefully reading my article and helpful comments.

\noindent B. Verkin Institute for Low Temperature Physics and Engineering\\
of the National Academy of Sciences of Ukraine\\
47, Nauky ave, Kharkiv, 61103, Ukraine

\medskip

\noindent Department of Mathematics  
University of Toronto \\
40 St. George Street
Toronto, ON,  M5S 2E4
Canada 

\medskip

\noindent e-mail:    gennadiy\_f@yahoo.co.uk

\end{document}